\documentclass{article}
\usepackage{amsmath}
\usepackage{amsthm}
\usepackage{amssymb}
\usepackage{dsfont}
\usepackage{calc}
\usepackage{eufrak} 
\usepackage{hyperref}
\usepackage{authblk}
\usepackage{theoremref}
\usepackage{verbatim}

\begin{document}
\bibliographystyle{alpha}
\newtheorem{theorem}{Theorem}
\newtheorem{lemma}{Lemma}
\newtheorem{definition}{Definition}
\newtheorem{proposition}{Proposition}
\newtheorem{remark}{Remark}
\newtheorem{property}{Property}
\newtheorem{corollary}{Corollary}
\newcounter{casenum}
\newenvironment{caseof}{\setcounter{casenum}{1}}{\vskip.5\baselineskip}
\newcommand{\case}[2]{\vskip.5\baselineskip\par\noindent {\bfseries Case \arabic{casenum}:} #1\\#2\addtocounter{casenum}{1}}

\theoremstyle{plain}		
				\newtheorem{thm}{Theorem}[section]
					\newtheorem*{convention}{Convention}
					\newtheorem*{acknowledgement}{Acknowledgement}

\title{Groups with decidable word problem that do not embed in groups with decidable conjugacy problem}
\author{Arman Darbinyan}
\date{}

\affil{\emph{}}

\maketitle

\begin{abstract}
	We show the existence of finitely presented torsion-free groups with decidable word problem that cannot be embedded in any finitely generated group with decidable conjugacy problem. This answers a well-known question of Collins from the early 1970's.
\end{abstract}

\section{Introduction}

In this article we construct torsion-free groups with decidable word problem that cannot be embedded in groups with decidable conjugacy problem, hence answering the following question first asked by Collins in the early 1970's.
 \begin{quote}
  Question 1. Can every torsion-free group with decidable word problem be embedded in a
group with decidable conjugacy problem?	
  \end{quote}
 Probably the first source where this question was posed in a written form is \cite{collins-problem-source-1}. This question also appears in the 1976 edition of The Kourovka Notebook as Problem 5.21, \cite{kourovka}.  Other sources with this question include \cite{collins-problem-source-2, olsh-sapir-survey, sapir-olsh-- conjugacy}.\\
 
Let $G=\langle X \rangle$ be a  group with a given finite generating set $X$. If we consider the elements of $X\cup X^{-1}$ as formal letters and denote by $(X\cup X^{-1})^*$ the set of finite words from the alphabet $X\cup X^{-1}$, then each finite word from $(X\cup X^{-1})^*$ in a natural way corresponds to an element in $G$,  and vice versa: for each element $g\in G$ there exists a word (in fact, countably many words) $w \in (X\cup X^{-1})^*$ that corresponds to $g$, in which case we write $w=_G 1$.  It is said that $G$  has decidable word problem if there exists an algorithm that for each $w\in (X\cup X^{-1})^*$ decides whether or not $w=_G 1$. It is said that $G$ has decidable conjugacy problems if there exists an algorithm that for each pair of words $u, v \in (X\cup X^{-1})^*$ decides whether or not $u$ and $v$ represent conjugate elements in $G$. Decidability of the conjugacy problem obviously implies the decidability of the word problem, nevertheless, in general, the inverse is not true \cite{miller-1, collins-wp-cp}. The word and conjugacy problems, along with the group isomorphism problem, are considered as the central decision problems in combinatorial and geometric group theory, and they were first studied by Max Dehn in 1911. See \cite{dehn}.

The connections between the word and conjugacy problems have been studied extensively by various authors. The first examples of finitely generated groups with decidable word problem but undecidable conjugacy problem were constructed by Miller III \cite{miller-1} and by Collins \cite{collins-wp-cp}. For other early results about connections between the word and conjugacy problems in groups see for example  \cite{collins-miller, gorjaga-kirk,  miller}.

 It was mentioned by Collins in \cite{kourovka} that, due to an example by Macintyre, there exists a finitely generated group with torsions and decidable word problem which cannot be embedded into a finitely generated group with decidable conjugacy problem.  However, the case for torsion-free groups remained  open until now. Indeed, one of the reasons why the torsion and torsion-free cases are different is based on the observation that conjugate elements in a group must have the same order, and since in a torsion-free group all non trivial elements have the same (infinite) order, for torsion-free groups one cannot make use of this observation to answer Question 1.

We resolve Question 1 with Theorem \ref{theorem-answer-to-collins} below.
\begin{theorem}
\label{theorem-answer-to-collins}
	There exists a finitely generated torsion-free group $\mathcal{G}$ with decidable word problem such that $\mathcal{G}$ cannot be embedded into a  group with decidable conjugacy problem. Moreover, $\mathcal{G}$ can be chosen to be a solvable group of derived length $4$ or a finitely presented group.
\end{theorem}
  
  In \cite{sapir-olsh-- conjugacy}, Olshanskii and Sapir showed that if in addition to the word problem the power problem is also solvable in a group, then the group embeds into a group with decidable conjugacy problem. Another evidence for the seemingly positive answer to the question of Collins was a well-known theorem of Osin stating that every torsion-free countable group embeds into a two-generated group with exactly two conjugacy classes \cite{osin annals}. In \cite{miasnikov-schupp}, Miasnikov and Schupp observed that every finitely generated group with recursively enumerable presentation and finite number of conjugacy classes has decidable conjugacy problem. It follows from Theorem \ref{theorem-answer-to-collins} and from the observation of Miasnikov and Schupp that Osin's theorem cannot be improved so that it preserves the property of having recursively enumerable presentation.
~\\

Two disjoint sets of natural numbers $S_1, S_2 \subset \mathbb{N}$ are called \textit{recursively inseparable} if there is no recursive set $T \subset \mathbb{N}$ such that $S_1 \subseteq T$ and $S_2 \subseteq \mathbb{N} \setminus T$. The set $T$ is called \textit{separating set}. Clearly, if two disjoint sets are recursively inseparable, then both of them are non-recursive. Indeed, if, say, $S_1$ is recursive, then as a recursive separating set  one could simply take $S_1$. Nevertheless, it is well-known that there exist recursively enumerable and recursively inseparable pairs of sets. See \cite{shoenfield-logic}.
One of the key ingredients in the proof of Theorem \ref{theorem-answer-to-collins} is the utilization of the recursively enumerable recursively inseparable sets. Up to our knowledge this is the second time after a result of Miller III \cite[Corollary 3.9]{miller} that this concept was utilized to produce a new result in a group theoretical setting. The method used for obtaining the main result of the current paper was later applied by the author to answer other open questions as well, \cite{darbinyan-orders}. 

\begin{acknowledgement}
 \emph{I would like  to thank Alexander Olshanskii for his thoughtful comments on this work.    }
\end{acknowledgement}




\section{Countable groups with decidable word problem}


Suppose that $G=\langle X \rangle$, where $X = \{ x_1, x_2, \ldots \}$. 
Then let us denote 
\begin{align*}
	WP_{X}(G)=\{w \in (X \cup X^{-1})^* \mid w=_G 1\}.
\end{align*}
We call the set $WP_X(G)$ the word problem of $G$ with respect to the enumerated generating set $X = \{ x_1, x_2, \ldots \}$. We say that $WP_X(G)$ is decidable if there exists an algorithm that for any input $(\pm n_1, \pm n_2, \ldots \pm  n_k)$, $k\in \mathbb{N}$, $n_i\in \mathbb{N}$, $i=1, \ldots k$, decides whether or not the word $x_{n_1}^{\pm 1}x_{n_2}^{\pm 1} \ldots x_{n_k}^{\pm 1}$ represents the trivial element of $G$. If $WP_X(G)$ is decidable, then we say that the word problem of $G$ is decidable with respect to the enumerated generating set $X$.

Let $i: \mathbb{N} \rightarrow \mathcal{I}$ be a bijection, where $\mathcal{I}$ is a recursive set. Then we say that the word problem in $G$ is decidable with respect to the enumeration $X=\{ x_{i(1)},  x_{i(2)}, \ldots \}$ (we denote $WP_{(X, i)}(G)$), if there exists an algorithm that for any input  $$\big(\pm i(n_1), \pm i(n_2), \ldots \pm  i(n_k)\big),$$
$k\in \mathbb{N}$, $n_i\in \mathbb{N}$, $i=1, \ldots k$,
decides whether or not the word  $x_{i(n_1)}^{\pm 1}x_{i(n_2)}^{\pm 1} \ldots x_{i(n_k)}^{\pm 1}$ represents the trivial element of $G$. In case $i: \mathbb{N} \rightarrow \mathcal{I}$ is computable, the decidability of $WP_{(X, i)}(G)$ is equivalent to the decidability of $WP_X(G)$. 
~\\

Traditionally, the word problem in groups is defined with respect to finite generating sets. The main advantage of the finite generating sets is that in this case the decidability of the word problem does not depend on a particular choice of generators and their enumeration, hence the decidability becomes a group intrinsic property rather than being dependent on the enumerated generating set. The groups with decidable word problems with respect to countable infinite enumerated sets were independently introduced in slightly different terms by Rabin \cite{rabin} and Mal'cev \cite{mal'cev} and are called in literature by the more common name of \emph{computable groups}. 

If a finitely generated group has decidable word problem with respect to a (equivalently, any) finite generating set, then it is said that the group has \emph{decidable word problem}.

 The following embedding theorem is from \cite{darbinyan}.
\begin{thm}[\cite{darbinyan}]
\label{th-embedding}
	Let $G = \langle X  \rangle$, $X=\{x_1, x_2, \ldots \}$, such that the word problem $WP_X(G)$ is decidable. Then there exists an embedding $\Phi_X: G \rightarrow H=\langle c, s \rangle$ of $G$ into a two generated group $H$ such that the following holds.
	\begin{enumerate}
		\item[(1)] The word problem is decidable  in $H$;
		\item[(2)] The map $n \mapsto \Phi_X(x_n)$ is computable (where $\Phi_X(x_n)$ is represented as a word from $\{c^{\pm 1}, s^{\pm 1}\}^*$);
		\item[(3)] If $G$ is torsion-free, then so is $H$;
		\item[(4)] If $G$ is a solvable group of derived length $l$, then $H$ is a solvable group of derived length $l+2$.
	\end{enumerate}
\end{thm}
\begin{remark}
    Parts (1) and (4) of Theorem \ref{th-embedding} are explicitly stated in Theorem 1 and Corollary 2 in \cite{darbinyan}. However, even though (2) and (3) are not stated explicitly, they follow from the fact that the embedding in \cite{darbinyan} embeds countable group $G = \langle X  \rangle$ (notations differ there) into a two-generated subgroup $H=\langle c, s \rangle$ of $G \wr \mathbb{Z} \wr \mathbb{Z}$ by $\Phi_X: x_n \mapsto [c, c^{s^{2^n-1}}]$ (see Page  5 in \cite{darbinyan}). The computability properties of the word problem in $H$ strongly depend on the generating set $X$ and its enumeration.
\end{remark}

\section{The construction}
In order to show the existence of $\mathcal{G}$ from Theorem \ref{theorem-answer-to-collins}, first, we will construct a  countable but not finitely generated group $\dot{G}$ with decidable word problem. Then $\mathcal{G}$ will be defined as a group in which $\dot G$ embeds in a certain way.\\

Let us fix two disjoint recursively enumerable and recursively inseparable sets $\mathcal{N}=\{n_1, n_2, \ldots \} \subset \mathbb{N}$ and $\mathcal{M}=\{m_1, m_2, \ldots \} \subset \mathbb{N}$ such that the maps $i \mapsto n_i$ and $i \mapsto m_i$ are computable.

For all $n \in \mathbb{N}$, define the group $A_n$ as a free abelian additive group of rank two with basis $\{a_{n, 0}, a_{n, 1}\}$, that is
$$A_n= \langle a_{n, 0} \rangle\bigoplus \langle a_{n, 1} \rangle,$$
and such that the groups $A_1, A_2, \ldots$ are disjoint.

For all $n \in \mathbb{N}$, define the groups $\dot{A}_n$ as follows:
\begin{equation*}
\dot{A}_n = \left\{
                      \begin{array}{ll}
                       A_{n}  & \mbox{if $n \notin (\mathcal{N} \cup \mathcal{M})$ ,}\\
                       A_n  / \ll a_{n, 1} = 2^{i} a_{n, 0} \gg& \mbox{if $n =n_i \in \mathcal{N}$ ,}\\
                       A_n  / \ll a_{n, 1} = 3^{i} a_{n, 0} \gg& \mbox{if $n =m_i \in \mathcal{M}$.}
                     \end{array}
                    \right.
\end{equation*}

For all $n \in \mathbb{N}$, define the group $B_n$ as a free  abelian additive group of rank $2$, that is
\begin{align*}
	B_n =  \langle b_{n, 0} \rangle \bigoplus \langle b_{n, 1} \rangle,
\end{align*}
such that $B_1$, $B_2$, \ldots ~are disjoint.

Now, for all $n\in \mathbb{N}$, define the groups $\dot{B}_n$ as follows:
 \begin{equation*}
\dot{B}_n = \left\{
                      \begin{array}{ll}
                       B_n & \mbox{if $n \notin (\mathcal{N} \cup \mathcal{M})$ ,}\\
                       B_{n} / \ll b_{n, 1} = 2^{i} b_{n, 0} \gg & \mbox{if $n=n_i \in \mathcal{N}$ or $n=m_i \in \mathcal{M}$.}
                     \end{array}
                    \right.
\end{equation*}
~\\

For all $n\in \mathbb{N}$ and $\epsilon\in \{0, 1\}$, let us denote the images of $a_{n, \epsilon}$ under the natural homomorphisms $A_n \rightarrow \dot{A}_n$ by $\dot{a}_{n, \epsilon}$. Analogously, the image of $b_{n, \epsilon}$ under the natural homomorphism $B_n \rightarrow \dot{B}_n$ we denote by $\dot{b}_{n, \epsilon}$. It follows from the definitions of $\dot{A}_n$ and $\dot{B}_n$ that, for all $n \in \mathbb{N}$, these groups are infinite and torsion free.

\begin{convention}
Below, whenever a group $G$ is given with a generating set $S$, depending on the context, we either consider $S\cup S^{-1}$ as a subset of $G$ or we consider it as a set of formal letters (i.e. alphabet) that can form finite words. For example, if $G$ is an additive abelian group and $a, b \in S$, then $ab^{-1}$ and $b^{-1}a$ are finite words over the alphabet $S\cup S^{-1}$ that represent the element $a-b \in G$. 
\end{convention}

\begin{lemma}
\label{lem-wp-An}
	
	There exists an algorithm such that for each input $n \in \mathbb{N}$ and $w  \in \{ \dot{a}_{n, 0}^{\pm 1},  \dot{a}_{n, 1}^{\pm 1} \}^*$, it decides whether or not $w$ represents the trivial element in the group $\dot A_n$. Analogous statement also holds for the groups $\dot B_n$ with respect to the generating sets $\{ \pm \dot{b}_{n, 0}, \pm \dot{b}_{n, 1} \}$.
\end{lemma}
\begin{proof}
Indeed, since $\dot{A}_n$ is abelian with generating set $\{\dot a_{n, 0}, \dot a_{n, 1} \}$, each word $w$ from $\big\{ \dot a_{n, 0}^{\pm 1}, \dot a_{n, 1}^{\pm 1} \big\}^*$ can be algorithmically transformed to a word of the form
	\begin{align*}
		w'=  \big(\dot a_{n, 0} \big)^{\lambda_0}  \big( \dot a_{n, 1} \big)^{\lambda_1}, ~\lambda_0, \lambda_1 \in \mathbb{Z}
	\end{align*}
that represents the same element of $\dot{A}_n$ as the initial word $w$.
	
	Now, assuming that $\lambda_0 \neq 0, \lambda_1 \neq 0$, in order for $w'$ to represent the trivial element in $\dot{A}_n$ it must be that $n \in \mathcal{N} \cup \mathcal{M}$, because if otherwise, then, by definition, the group $\dot{A}_n$ is free abelian of rank $2$ with basis $\{\dot a_{n, 0}, \dot a_{n, 1} \}$.
	
	In case $n \in \mathcal{N}$, by definition we have that $\dot a_{n, 1} = 2^{x} \dot a_{n, 0}$, where $x$ is the index of $n$ in $\mathcal{N}$, i.e. $n=n_x$.	Similarly, in case $n \in \mathcal{M}$, by definition we have that $\dot{a}_{n, 1} = {3^{x} }\dot a_{n, 0}$, where $x$ is the index of $n$ in $\mathcal{M}$, i.e. $n=m_x$.
	
	Now, if $\lambda_0 = 0$ and $\lambda_1 = 0$, then clearly $w'$ (hence also $w$) represents the trivial element in $\dot A_n$. Therefore,  without loss of generality let us assume that at least one of $\lambda_0$ and $\lambda_1$ is not $0$. Then,
	if we treat $x$ as an unknown variable,  depending on whether $n= n_x$ or $n=m_x$, the equality $w'=0$ would imply one of the following equations:
	\begin{align}
	\label{eq - 3.1}
		\lambda_0 + \lambda_1 2^{x} =0
	\end{align}
	or
	\begin{align}
	\label{eq - 3.2}
		\lambda_0 + \lambda_1 3^{x} =0,
	\end{align}
	respectively.
	
	This observation suggests that in case $\lambda_0 \neq 0 $ or $\lambda_1 \neq 0$, in order to verify whether or not $w'=0$ in $\dot{A}_n$, we can first try to find $x$ satisfying  \eqref{eq - 3.1} or  \eqref{eq - 3.2}, and in case such an $x$ does not exist, conclude that $w'$ (hence, also $w$) does not represent the trivial element in $\dot{A}_n$. Otherwise, if $x$ is the root of the equation \eqref{eq - 3.1}, we can check whether or not $n = n_x$ (since $\mathcal{N}$ is recursively enumerable, this checking can be done algorithmically). Similarly, if $x$ is the root of the equation \eqref{eq - 3.2}, we can check whether or not $n = m_x$.	If as a result of this checking we get $n=n_x$ (respectively, $n=m_x$), then the conclusion will be that $w'$ (hence, also $w$) represents the trivial element in $\dot{A}_n$, otherwise, if $n \neq n_x$ (respectively, $n \neq m_x$), then the conclusion will be that $w'$ (hence, also $w$) does not represent the trivial element in $\dot{A}_n$.

The proof of the same statement for the groups $\dot{B}_n$ is identical.

\end{proof}

Define 
\begin{equation*}
	\dot A= \bigoplus_{n=1}^{\infty} \dot{A}_n, ~ \dot B = \bigoplus_{n=1}^{\infty} \dot{B}_n,  ~\dot{G}_0 = \dot{A} \bigoplus \dot{B}\text{~and~} C=\bigoplus_{n=1}^{\infty} \langle c_n \rangle,
\end{equation*}
where $\langle c_1 \rangle, \langle c_2 \rangle, \ldots$ are pairwise disjoint infinite cyclic groups.
Denote $$S_0=\{ \dot a_{n, 0}, \dot a_{n, 1}, \dot b_{n, 0}, \dot b_{n, 1} \mid n\in \mathbb{N} \}$$
and
$$S_1=\{c_1, c_2, \ldots \}.$$
\begin{lemma}
\label{lemma-wp-in-G-0}
The word problems $WP_{S_0}(\dot{G}_0)$ and $WP_{S_1}(C)$ are decidable.	
\end{lemma}
\begin{proof}
Let us denote $S_{0,a}=\{\dot{a}_{n, 0}, \dot{a}_{n, 1} \mid n \in \mathbb{N} \}$ and 	$S_{0,b}=\{\dot{b}_{n, 0}, \dot{b}_{n, 1} \mid n \in \mathbb{N} \}$.  Since by definition $\dot{G}_0$ is  the direct sum of $\dot{A}$ and $\dot{B}$, the decidability of  $WP_{S_0}(\dot{G}_0)$ is equivalent to the decidability of both $WP_{S_{0, a}}(\dot{A})$ and  $WP_{S_{0, b}}(\dot{B})$. On its own turn, the decidability of  $WP_{S_{0, a}}(\dot{A})$ and  $WP_{S_{0, b}}(\dot{B})$ is an immediate consequence of Lemma \ref{lem-wp-An}.

The decidability of $WP_{S_1}(C)$ is obvious.
\end{proof}
~\\

Let $$\psi: C \rightarrow Aut(\dot G_0)$$ such that for $i \in \mathbb{N}$, $\psi(c_i)$ is the automorphism of $\dot G_0$ that fixes all elements from $\dot G_0$ except the ones from $\dot{A}_{n_i}\bigoplus \dot{B}_{n_i} \leq \dot G_0$,  and  for $\epsilon \in \{0, 1\}$, it flips $\dot{a}_{n_i, \epsilon}$ and $\dot{b}_{n_i, \epsilon}$, that is $\psi(c_i): \dot{a}_{n_i, \epsilon} \mapsto \dot{b}_{n_i, \epsilon}$ and $\psi(c_i): \dot{b}_{n_i,\epsilon} \mapsto \dot{a}_{n_i, \epsilon}$. The elements $\psi(c_i) \in Aut(\dot{G}_0)$, $i\in \mathbb{N}$, pairwise commute.  Thus $\psi$ is the induced homomorphism from $C$ to $Aut(\dot{G}_0)$. Now define $\dot{G}$ as the semidirect product of $\dot{G_0}$ and $C$ with respect to $\psi$, that is
\begin{align*}
	\dot G = \dot G_0 \rtimes_{\psi} C.
\end{align*}

Denote $$S=\{\dot a_{n, 0}, \dot a_{n, 1}, \dot b_{n, 0}, \dot b_{n, 1}, c_n \mid n\in \mathbb{N} \}.$$
\begin{lemma}
\label{lem-clarifying}
	The group $\dot G$ is torsion-free and metabelian. Also, for every $i, n \in \mathbb{N}$ and $s\in S_0$, we have
	
\begin{align*}
	    c_i^{-1} s c_i=c_i s c_i^{-1}=\left\{
                      \begin{array}{ll}
                   \dot{b}_{n_i, \epsilon}  & \mbox{if $\epsilon\in \{0, 1\}$, $n =n_i \in \mathcal{N}$  and $s = \dot{a}_{n_i, \epsilon}$ }\\
                   \dot{a}_{n_i, \epsilon}  & \mbox{if $\epsilon \in \{0, 1\}$, $n =n_i \in \mathcal{N}$  and $s = \dot{b}_{n_i, \epsilon}$ }\\
                   s  & \mbox{otherwise.}
                     \end{array}
                    \right.
\end{align*}
\end{lemma}
\begin{proof}
	The group $\dot G$ is a semidirect product of two torsion-free abelian groups, hence it is torsion-free and metabelian. The second statement follows from the definitions of the semidirect product, the automorphism $\psi(c_i): \dot{G}_0 \rightarrow \dot{G}_0 $, and from the identities $\dot a_{n_i, 1} = 2^i \dot a_{n_i, 0}$,  $\dot b_{n_i, 1} = 2^{i}  \dot b_{n_i, 0}$ and $c_i s c_i^{-1} = \psi(c_i)(s)$, $c_i^{-1} s c_i = \psi(c_i)^{-1}(s) =\psi(c_i)(s)$.
\end{proof}

\begin{lemma}
	The word problem $WP_S(\dot{G})$ is decidable.
	
\end{lemma}
\begin{proof}
Note that every word from $(S \cup S^{-1})^*$ can be algorithmically transformed into a word of the form $ e_1^{u_1}\ldots e_k^{u_k} u_{k+1}$ that represents the same element of $\dot{G}$, where $ e_1, \ldots, e_k \in S_0\cup S_0^{-1}$ and $u_1, \ldots, u_{k+1} \in (S_1\cup S_1^{-1})^*$.
On the other hand, by Lemma \ref{lem-clarifying}, for all $i\in \mathbb{N}$ and $s\in S_0 \cup S_0^{-1}$, $c_i^{\pm 1} s c_i^{\mp 1}$ is equal to some $\bar{s} \in S_0 \cup S_0^{-1}$. Moreover, since the sets $\mathcal{N}=\{n_1, n_2, \ldots \}$ and $\mathcal{M}=\{m_1, m_2, \ldots \}$ are recursively enumerated, $\bar{s}$ can be algorithmically verified. Therefore, every word of the form $ e_1^{u_1}\ldots e_k^{u_k} u_{k+1}$, on its own turn, can be algorithmically transformed into a word of the form $ e'_1\ldots e'_k u_{k+1}$ that represents the same element of $\dot{G}$ as the initial word, where $e'_1,\ldots, e'_k  \in S_0\cup S_0^{-1}$. Finally, the words of the form $ e'_1\ldots e'_k u_{k+1}$ represent the trivial element of $\dot{G}$ if and only if the word $e'_1\ldots e'_k$ represents the trivial element of $\dot{G}_0$ and $u_{k+1}$ represents the trivial element of $C$. By Lemma \ref{lemma-wp-in-G-0}, the last conditions can be checked algorithmically. Therefore, the word problem $WP_S(\dot{G})$ is decidable.

\end{proof}

Let us introduce the  enumeration $i: S \rightarrow \mathbb{N}$ of the set $S$:
\begin{align*}
    i(s)=\left\{
                      \begin{array}{ll}
                       5(n-1)+1  & \mbox{if $s=\dot{a}_{n, 0}$,}\\
                       5(n-1)+2  & \mbox{if $s=\dot{a}_{n, 1}$,}\\
                       5(n-1)+3  & \mbox{if $s=\dot{b}_{n, 0}$,}\\
                       5(n-1)+4  & \mbox{if $s=\dot{b}_{n, 1}$,}\\
                       5n & \mbox{if $s=c_n$.}\\
                     \end{array}
                    \right.
\end{align*}
Denote $X=\{x_1, x_2, \ldots \}$, where $x_n = i^{-1}(n)$, $n\in \mathbb{N}$. ($X$ and $S$ coincide as sets but have differing enumerations.) The essential property of $i$ is that it is a computable bijection, which implies that $WP_X(\dot{G})$ is decidable.

Now suppose that $\Phi = \Phi_X: \dot{G} \hookrightarrow H$ is an embedding of the group $\dot{G}$ into a two-generated torsion-free group $H$ such that it satisfies the properties from Theorem \ref{th-embedding}. In particular, the maps  $\phi_1$ and $\phi_2$ defined as

\begin{align*}
\phi_1:(n, \epsilon) \mapsto  \Phi(\dot{a}&_{n, \epsilon}),~ \phi_2:(n, \epsilon) \mapsto  \Phi(\dot{b}_{n, \epsilon}), \text{~and~}\phi_3: n \mapsto  \Phi(c_n),\\
~\\
  & \text{where~} n\in \mathbb{N}, \epsilon \in \{0, 1\},
\end{align*}
are computable, 
 and $H$ has decidable word problem. As the next lemma shows, the group $H$ has  the desirable properties we were looking for.
\begin{lemma}
	\label{final lemma}
	The group $H$ cannot be embedded in a finitely generated group with decidable conjugacy problem.
\end{lemma}

\begin{proof}
	By contradiction, let us assume that $H$ embeds in a (finitely generated) group $\tilde{H}$ that has decidable conjugacy problem. Then, for the purpose of convenience, without loss of generality let us assume that $H$ is  a subgroup of the group $\tilde{H}$.
	
	 Below we  show that the decidability of the conjugacy problem in $\tilde{H}$ contradicts  the assumption that $\mathcal{N}$ and  $\mathcal{M}$ are recursively inseparable.
	
	Let us define $\mathcal{C} \subseteq \mathbb{N}$ as
	\begin{align*}
		\mathcal{C}=\big\{n \in \mathbb{N} \mid \Phi(\dot{a}_{n, 0}) \text{~is conjugate to~} \Phi(\dot b_{n, 0}) \text{~in~} \tilde{H}\big\}.
	\end{align*}
Since the above described maps $\phi_1$, $\phi_2$ and $\phi_3$ are computable and  $\tilde{H}$ has decidable conjugacy problem, there exists an algorithm that for any $n\in \mathbb{N}$ verifies whether or not  $\Phi(\dot a_{n, 0})$ is conjugate to $\Phi(\dot b_{n, 0})$ in $\tilde{H}$. Therefore, the set $\mathcal{C}$ is recursive.

Note that for all $i\in \mathbb{N}$, since by Lemma \ref{lem-clarifying} the identity $c_i^{-1} \dot{a}_{n_i, 0} c_i = \dot{b}_{n_i, 0}$ holds, we get that $\Phi( \dot{a}_{n_i, 0})$ is conjugate to $\Phi(\dot{b}_{n_i, 0})$ in $\tilde{H}$. Therefore, $\mathcal{N}\subseteq \mathcal{C}$.

	Since for groups with decidable conjugacy problem one can algorithmically find conjugator element for each pair of conjugate elements of the group, recursiveness of $\mathcal{C}$ implies that there exists a computable map
	\begin{align*}
		{f}: \mathcal{C} \rightarrow \tilde{H}
	\end{align*}
	such that for all $n\in \mathcal{C}$ we have  
	$$f(n)^{-1} \Phi(\dot{a}_{n, 0}) f(n) = 	\Phi(\dot b_{n, 0}).$$
	(For example, one can define $f(n)$ as the element of $\tilde H$ that corresponds to the lexicographically smallest conjugator word composed by finite generator letters of $\tilde H$.)
	For $n \in \mathcal{C}$, let us denote
	$$f(n)=h_n \in \tilde{H}.$$
	
		Now let us define
	\begin{align*}
		\mathcal{A} = \big\{ n \in \mathcal{C} \mid  h_n^{-1} \Phi(\dot{a}_{n, 1}) h_n = 	\Phi(\dot b_{n, 1}) \big \} \subseteq \mathbb{N}.
	\end{align*}
Since  the word problem in $\tilde{H}$ is decidable, the set $\mathcal{C}$ is recursive and the maps $\Phi$ and $f$ are computable, we get that the set $\mathcal{A}$ is a recursive subset of $\mathbb{N}$. Also since 
\begin{align*}
	\dot a_{n_i, 1} = 2^i \dot a_{n_i, 0}, \dot b_{n_i, 1} = 2^{i}  \dot b_{n_i, 0} \text{~and~} c_i^{-1}\dot a_{n_i, 0} c_i =\dot b_{n_i, 0}, \text{~for~} i\in \mathbb{N},  
\end{align*}
in $\dot G$, we get that each conjugator of the pair $\Phi(\dot{a}_{n_i, 0}), \Phi(\dot{b}_{n_i, 0})$ is also a conjugator for the pair $\Phi(\dot{a}_{n_i, 1}), \Phi(\dot{b}_{n_i, 1})$. 
Therefore, since $ \mathcal{N} \subseteq \mathcal{C}$, we get  $$ \mathcal{N}  \subseteq \mathcal{A} .$$
 On the other hand, since for each $ m_i \in \mathcal{M}$ we have
 \begin{align*}
	\dot a_{m_i, 1} = 3^i \dot a_{m_i, 0}, ~\dot b_{m_i, 1} = 2^{i}  \dot b_{m_i, 0},  
\end{align*}
we get that the pairs of elements
\begin{align*}
	\big(\Phi(\dot a_{m_i, 0}), ~\Phi(\dot b_{m_i, 0}) \big)
	\text{~and~}
    \big(\Phi(\dot a_{m_i, 1}), ~\Phi(\dot b_{m_i, 1}) \big)
\end{align*}
cannot be conjugated in $\tilde{H}$ by the same conjugator. Therefore, we get that 
$$\mathcal{A} \cap \mathcal{M} = \emptyset.$$

Thus we got that $\mathcal{A} \subset \mathbb{N}$ is a recursive set such that $\mathcal{N} \subseteq \mathcal{A}$ and $\mathcal{A} \cap \mathcal{M} = \emptyset$. However, this contradicts the assumption that $\mathcal{N}$ and $\mathcal{M}$ are recursively inseparable. Lemma \ref{final lemma} is proved.

	\end{proof}

Finally, note that since $\dot G$ is metabelian, by property (4) of Theorem \ref{th-embedding}, $H$ is a solvable group of derived length $4$. Therefore, Lemma \ref{final lemma} asserts that there exists a solvable group of derived length $4$ that satisfies the statement of Theorem \ref{theorem-answer-to-collins}.
Also by a version of Higman's embedding theorem described by Aanderaa and Cohen in \cite{aanderaa-wp}, the group $H$ can be embedded into a finitely presented group $\mathcal{G}$ with decidable word problem. As Chiodo and Vyas showed in \cite{chiodo-vyas}, the group $\mathcal{G}$ defined this way will  also inherit the property of torsion-freeness from the group $H$. 

Since $H$ cannot be embedded into a group with decidable conjugacy problem, this property will be inherited by $\mathcal{G}$. Thus Theorem \ref{theorem-answer-to-collins} is proved.

~\\
~\\
 A. Darbinyan, \textsc{Department of Mathematics, Texas A\&M,
    Blocker Building, 3368 TAMU, 155 Ireland Street, College Station, TX, USA 77840}\par\nopagebreak
  \textit{E-mail address}: \texttt{adarbina@math.tamu.edu}


\begin{thebibliography}{20}


\bibitem{aanderaa-wp}
S. Aanderaa, D. E. Cohen, Modular machines I, II, in [2], Stud. Logic Found. Math. 95 (1980),
1-18, 19-28.
\bibitem{collins-problem-source-2}
		 G. Baumslag, A. Myasnikov,  V. Shpilrain et al., Open problems in Combinatorial and Geometric
Group Theory, http://www.grouptheory.info/.



\bibitem{collins-problem-source-1}
		 W.W. Boone,  F.B. Cannonito, and  R.C. Lyndon, Word Problems: Decision
Problems and the Burnside Problem in Group Theory, Studies in Logic and
the Foundations of Mathematics, vol. 71, North-Holland, Amsterdam, 1973.

\bibitem{chiodo-vyas}
M. Chiodo, R. Vyas, Torsion length and finitely presented groups, J Group Theory 21:5, (2018), 947-969.


 \bibitem{collins-wp-cp}
 D.J. Collins, Representation of Turing reducibility by word and conjugacy problems in
finitely presented groups, Acta Math. 128 , (1972)  no. 1-2,73-90. 

  
\bibitem{collins-miller}
D.J. Collins, C.F. Miller III , The conjugacy problem and subgroups of finite index, Proc.
London Math. Soc. (3) 34 (1977), no. 3, 535-556.

\bibitem{dehn}
M. Dehn, Über unendliche diskontinuierliche Gruppen, Mathematische Annalen 71, (1911), 116-144.

\bibitem{darbinyan}
A. Darbinyan, Group embeddings with algorithmic properties, Communications in Algebra, 43:11  (2015), 4923-4935.
\bibitem{darbinyan-orders}
A. Darbinyan, Computability, order and solvable groups, arXiv:1909.05720





\bibitem{gorjaga-kirk}
A.V. Gorjaga, A.S. Kirkinskii, The decidability of the conjugacy problem cannot be transferred to finite extensions of groups. (Russian) Algebra i Logika 14 (1975), no. 4, 393-406.


\bibitem{kourovka}
	 Kourovka Notebook. Unsolved Problems in Group Theory. 5th edition, Novosibirsk, 1976.




\bibitem{mal'cev}
A. Mal'cev, Constructive algebras. I, Uspehi Mat. Nauk, vol. 16 (1961), no. 3 (99),
pp. 3-60.

\bibitem{miasnikov-schupp}
	A. Miasnikov, P. Schupp, Computational complexity and the conjugacy problem, Computability 6 (2017), 307-318.
 
\bibitem{miller-1}
C. F. Miller III. On group-theoretic decision problems and their classification, volume 68
of Annals of Mathematics Studies. Princeton University Press, 1971.

\bibitem{miller}
C.F. Miller III, Decision problems for groups: a survey and reflections, in Algorithms and classification
in combinatorial group theory (Berkeley, CA, 1989),  Math. Sci. Res. Inst.
Publ., Vol. 23, Springer, 1992.

	
\bibitem{olsh-sapir-survey}
A. Yu. Olshanskii, M. V. Sapir, Length and area functions on groups and quasi-isometric Higman
embeddings, Internat. J. Algebra Comput., 11(2), (2001), 137-170.

  \bibitem{sapir-olsh-- conjugacy}
A. Yu. Olshanskii, M.V. Sapir,  Subgroups of finitely presented groups with
solvable conjugacy problem. Internat. J. Algebra Comput., 15(5-6) (2005), 1075-1084.

\bibitem{osin annals}
		D. Osin, Small cancellations over relatively hyperbolic groups and embedding theorems,
		Ann. Math. 172 (2010), no. 1, 1-39.


 
 \bibitem{rabin}
  M. Rabin, Computable algebra, general theory and theory of computable fields,
Trans. Amer. Math. Soc., vol. 95 (1960), pp. 341-360.
	
\bibitem{shoenfield-logic}
 R. M. Smullyan. Undecidability and recursive inseparability. Mathematical Logic Quarterly,  4(7-11),  (1958), 143-147.


\end{thebibliography}
\end{document}